\documentclass[12pt]{article}
\usepackage{amssymb}
\usepackage{verbatim}
\usepackage{array}
\usepackage{latexsym}
\usepackage{enumerate}
\usepackage{amsmath}
\usepackage{amsfonts}
\usepackage{amsthm}
\usepackage{graphicx,epsfig}
\usepackage[english]{babel}

\title{$k$-nets embedded in a projective plane over a field}
\date{}

\author{G.~Korchm\'aros\thanks{the research was performed while the first author was a visiting professor at the Bolyai Institute of University of Szeged during the second semester of the academic year 2011-12. The visit was financially supported by the TAMOP-4.2.1/B-09/1/KONV-2010-0005 project,}\,\, G.~P.~Nagy\thanks{the research was supported by the European Union and co-funded by the European Social Fund; project number  TAMOP-4.2.2.A-11/1/KONV-2012-0073}\,\,
and N.~Pace\thanks{the research was supported by FAPESP (Funda\c{c}\~ao de Amparo a Pesquisa do Estado de S\~{a}o Paulo), proc. no. 12/03526-0.}.}

\newtheorem{theorem}{Theorem}[section]
\newtheorem{proposition}[theorem]{Proposition}
\newtheorem{lemma}[theorem]{Lemma}

{\theoremstyle{definition}
\newtheorem*{definition*}{Definition}
\newtheorem{example}[theorem]{Example}

\newtheorem*{proposition*}{Proposition}
\newtheorem*{corollary*}{Corollary}
\newtheorem*{lemma*}{Lemma}

\def\cH{\mathcal H}

\def\cR{\mathcal R}

\def\cT{\mathcal T}

\def\cW{\mathcal W}

\begin{document}
\maketitle
\begin{abstract}
We investigate $k$-nets with $k\geq 4$ embedded in the projective plane $PG(2,\mathbb{K})$ defined over a field $\mathbb{K}$; they are line configurations in $PG(2,\mathbb{K})$ consisting of $k$ pairwise disjoint line-sets, called components, such that any two lines from distinct families are concurrent with exactly one line from each component. The size of each component of a $k$-net is the same,  the order of the $k$-net. If $\mathbb{K}$ has zero characteristic, no embedded $k$-net for $k\geq 5$ exists; see \cite{sj2004,ys2007}. Here we prove that this holds true in positive characteristic $p$ as long as $p$ is sufficiently large compared with the order of the $k$-net. Our approach, different from that used in \cite{sj2004,ys2007}, also provides a new proof in characteristic zero.
\end{abstract}

\section{Introduction}
\label{problem}
An (abstract) $k$-net is a point-line incidence structure whose lines are partitioned in $k$ subsets, called components, such that any two lines from distinct components are concurrent with exactly one line from each component. The components have the same size, called the order of the $k$-net and denoted by $n$. A $k$-net has $n^2$ points and $kn$ lines. A $k$-net (embedded) in $PG(2,\mathbb{K})$ is a subset of points and lines such that the incidence structure induced by them is a $k$-net.

In the complex plane, there are known plenty of examples and even infinity families of $3$-nets but only one $4$-net up to projectivity; see \cite{sj2004,urzua2009,ys2004,ys2007}. This $4$-net, called the classical $4$-net, has order $3$ and it exists since $PG(2,\mathbb{C})$ contains an affine subplane $AG(2,\mathbb{F}_3)$ of order $3$, unique up to projectivity, and the four parallel line classes of $AG(2,\mathbb{F}_3)$ are the components of a $4$-net in $PG(2,\mathbb{C})$. By a result of Stipins \cite{sj2004}, see also \cite{ys2007}, no $k$-net with $k\geq 5$ exists in $PG(2,\mathbb{C})$. Since Stipins' proof works over any algebraically closed field of characteristic zero, his result holds true in $PG(2,\mathbb{K})$ provided that $\mathbb{K}$ has zero characteristic.

Our present investigation of $k$-nets in $PG(2,\mathbb{K})$ includes groundfields $\mathbb{K}$ of positive characteristic $p$, and as a matter of fact, many more examples. This phenomena is not unexpected  since $PG(2,\mathbb{K})$ with $\mathbb{K}$ of characteristic $p>0$ contains an affine subplane $AG(2,\mathbb{F}_p)$ of order $p$ from which $k$-nets for $3\leq k \leq p+1$ arise taking $k$ parallel line classes as components. Similarly, if $PG(2,\mathbb{K})$ also contains an affine subplane  $AG(2,\mathbb{F}_{p^h})$, in particular if $\mathbb{K}=\mathbb{F}_q$ with $q=p^r$ and $h|r$, then $k$-nets of order $p^h$ for $3\le k \leq p^h+1$ exist in $PG(2,\mathbb{K})$. Actually, more families of $k$-nets in $PG(2,\mathbb{F}_q)$ when $q=p^r$ with $r\geq 3$ exist; see Example \ref{exam:Lun}.
On the other hand, no $5$-net of order $n$ with $p>n$ is known to exist. This suggests that for sufficiently large $p$ compared with $n$, Stipins' theorem remains valid in $PG(2,\mathbb{K})$. Our Theorem \ref{fotetel} proves it for $p>3^{\varphi(n^2-n)}$ where $\varphi$ is the classical Euler $\varphi$ function, and in particular for $p>3^{n^2/2}$. Our approach
also works in zero characteristic and provides a new proof for Stipins' result.

A key idea in our proof is to consider the cross-ratio of four concurrent lines from different components of a $4$-net. Proposition \ref{pr:constant} states that the cross-ratio remains constant when the four lines vary without changing component. In other words,  every $4$-net in $PG(2,\mathbb{K})$ has constant cross-ratio. By Theorem \ref{thenov10ter} in zero charactersitic, and by Theorem \ref{th10nov} in characteristic $p$ with $p>3^{\varphi(n^2-n)}$, the constant cross-ratio is restricted to two values only, namely to the roots of the polynomial $X^2-X+1$. From this,  the non-existence of $k$-nets for $k\ge 5$ easily follows both in zero characteristic and in characteristic $p$ with $p>3^{\varphi(n^2-n)}$. It should be noted that without a suitable hypothesis on $n$ with respect to $p$,  the constant cross-ratio of a $4$-net may assume many different values, even for finite fields, see Example \ref{exam:Lun}.

In $PG(2,\mathbb{K})$, $k$-nets naturally arise from pencils of curves, the components of the $k$-net being the completely reducible curves in the pencil. This has given a motivation for the study of $k$-nets in Algebraic geometry; see \cite{Dolgachev}, and \cite{ys2004}. We discuss this relationship in Section \ref{alggeo} and state an equation that will be useful in Section \ref{inv}.

\section{$k$-nets and completely irreducible curves in a pencil of curves}
\label{alggeo}
Let $\lambda_1,\lambda_2,\lambda_3$ be three components of a $k$-net of order $n$ embedded in $PG(2,\mathbb{K})$. Let $r_i=0$, $w_i=0$, $t_i=0$ ($i=1,\ldots,n$) be the equations of the lines in  $\lambda_1,\lambda_2,\lambda_3$, respectively. The completely reducible polynomials $R=r_1\cdots r_n$, $W=w_1\cdots w_n$ and $T=t_1\cdots t_n$ define three plane curves of degree $n$, say $\cR$, $\cW$ and $\cT$. Consider the pencil $\Lambda$ generated by $\cR$ and $\cW$. Since $\lambda_1,\lambda_2,\lambda_3$ are the components of a $3$-net of order $n$, there exist $\alpha,\beta\in \mathbb{K}^*$ such that $\cT$ and the curve $\cH$ of $\Lambda$ with equation $\alpha R+ \beta W=0$ have $n^2+1$ common points but no common components. From B\'ezout's theorem, $\cT=\cH$. Therefore,
\begin{equation}
\label{7geneqq1}
 \alpha r_1\cdots r_n + \beta w_1\cdots w_n + \gamma t_1\cdots t_n=0
 \end{equation}
holds for a homogeneous triple  $(\alpha,\beta,\gamma)$ with coordinates $\mathbb{K}^*$. Changing the projective coordinate system in $PG(2,\mathbb{K})$ the equations of the lines in the components of the $3$-net change but the homogeneous triple $(\alpha,\beta,\gamma)$ remains invariant.

Conversely, assume that an irreducible pencil $\Lambda$ of plane curves of degree $n$ contains $k$ curves each splitting into $n$ distinct lines, that is, $k$ completely reducible curves. Let $\lambda_i$ with $1\le i \le k$ be the set of the $n$ lines which are the factors of a completely reducible curve. Then $\lambda_1,\lambda_2,\ldots,\lambda_k$ are the components of a $k$-net embedded in $PG(2,\mathbb{K})$.

\section{The invariance of the cross-ratio of a $4$-net}
\label{inv}
Consider a $4$-net of order $n$ embedded in $PG(2,\mathbb{K})$ and label their components with $\lambda_i$ for $i=1,2,3,4$. We say that the $4$-net $\lambda=(\lambda_1,\lambda_2,\lambda_3,\lambda_4)$ has {\em{constant cross-ratio}} if for every point $P$ of $\lambda$ the cross-ratio $(\ell_1,\ell_2,\ell_3,\ell_4)$ of the four lines $\ell_i\in\lambda_i$ through $P$ is constant.

\begin{proposition} \label{pr:constant}
Every $4$-net in $PG(2,\mathbb{K})$ has constant cross-ratio.
\end{proposition}
\begin{proof}
In a projective reference system, let $r_i=0$, $w_i=0$, $t_i=0$, $s_i=0$ with $1\le i \le n$ be the lines of a $4$-net $\lambda=(\lambda_1,\lambda_2,\lambda_3,\lambda_4)$ respectively. Then there exist $\alpha,\beta,\gamma\in \mathbb{K}^*$ such that (\ref{7geneqq1}) holds
and $\alpha',\beta',\gamma'\in \mathbb{K}$ such that
\begin{equation} \label{7geneqq2}
\alpha' r_1r_2\cdots r_n+\beta' w_1w_2\cdots w_n+\gamma' s_1s_2\cdots s_n=0.
\end{equation}
As observed in Section \ref{alggeo}, the coefficients $\alpha,\beta,\gamma,\alpha',\beta',\gamma'$ remain invariant when the reference system is changed. Take a point $P$ of $\lambda$ and relabel the lines of $\lambda$ such that $r_1=0$, $w_1=0$, $t_1=0$ and $s_1=0$ are the four lines of $\lambda$ passing through $P$. We temporarily introduce the notation $(x_1, x_2, x_3)$ for the homogeneous coordinates of a point, and we arrange the reference system in such a way that $P$ coincides with the point $(0,0,1)$, the line $x_3=0$ contains no point from $\lambda_1$ or $\lambda_2$ while $r_1=x_1$ and $w_1=x_2$. Also, non-homogeneous coordinates $x=x_1/x_3$ and $y=x_2/x_3$ can be used so that $r_1=x$ and $w_1=y$. Note that we have arranged the coordinates so that $r_i,w_i,t_i,s_i$ have a zero constant term if and only if $i=1$. Let
\[\rho=\prod_{i=2}^n r_i(0,0),\hspace{5mm}
\omega=\prod_{i=2}^n w_i(0,0),\hspace{5mm}
\tau=\prod_{i=2}^n t_i(0,0),\hspace{5mm}
\sigma=\prod_{i=2}^n s_i(0,0).\]
Observe that
\[0=\alpha r_1\cdots r_n+\beta w_1\cdots w_n+\gamma t_1\cdots t_n=\alpha \rho x+ \beta \omega y+\gamma \tau t_1+ [\cdots],\]
where $[\ldots]$ stands for the sum of terms of degree at least $2$. From (\ref{7geneqq1}),
\[\frac{\alpha \rho}{\gamma\tau} x+ \frac{\beta \omega}{\gamma\tau} y+t_1=0.\]
Similarly,
\[\frac{\alpha' \rho}{\gamma'\sigma} x+ \frac{\beta' \omega}{\gamma'\sigma} y+s_1=0.\]
Therefore, the cross-ratio of the lines of $\lambda$ passing through $P$ is equal to
\begin{equation}
\label{7geneq3}
\kappa=\frac{\alpha \beta'}{\alpha' \beta}
\end{equation}
and hence it is independent of the choice of the point $P$.
\end{proof}
As an illustration of Proposition \ref{pr:constant} we compute the constant cross-ratio of the known $4$-net embedded in the complex plane.
\begin{example} Let $n=3$, and take a primitive third root of unity $\xi$. In homogeneous coordinates $(x,y,z)$ of $PG(2,\mathbb{K})$, let
\begin{align*}
r_1 &:=x,& r_2 &:=y, & r_3 &:=z,\\
w_1 &:=x+y+z, & w_2 &:=x+\xi y+\xi^2 z,& w_3 &:=x+\xi^2 y+\xi z,\\
t_1 &:=\xi x+y+z, & t_2 &:=x+\xi y+z, & t_3 &:=x+y+\xi z, \\
s_1 &:=\xi^2 x+y+z,& s_2 &:=x+\xi^2 y+z,& s_3 &:=x+y+\xi^2 z.
\end{align*}
Then these lines form a $4$-net $\lambda$ order $3$. Moreover,
\begin{align*}
t_1t_2t_3&=3(2\xi+1)r_1r_2r_3+\xi w_1w_2w_3,\\
s_1s_2s_3&=-3(2\xi+1)r_1r_2r_3+\xi^2 w_1w_2w_3.
\end{align*}
Hence, the constant cross-ratio of $\lambda$ is $\kappa=-1/\xi$.
\end{example}

\section{Some constraints on the constant cross-ratio of a $4$-net} It is well known that the cross-ratio of four distinct concurrent lines can take six possible different values depending on the order in which the lines are given. If $\kappa$ is one of them then $\kappa\neq 0,1$ and these six cross-ratios are
$$\kappa,\quad \frac{1}{\kappa},\quad 1-\kappa,\quad \frac{1}{1-\kappa},\quad \frac{\kappa}{\kappa-1},\quad 1-\frac{1}{\kappa}.$$
It may happen, however, that some of these values coincide, and this is the case if and only if either $\kappa \in \{-1,1/2,2\}$, or
\begin{equation}
\label{eq9nov}
\kappa^2-\kappa+1=0.
\end{equation}
Proposition \ref{pr:constant} says that the cross-ratio of four concurrent lines of a $4$-net takes the above six values for a given $\kappa\neq 0,1$, and each of these values can be considered as the \emph{constant cross-ratio} of the $4$-net. Now, the problem consists in computing $\kappa$.
We are able to do it in zero characteristic showing that $\kappa$ satisfies Equation (\ref{eq9nov}). In positive characteristic there are more possibilities. This will be discussed after proving the following result.
\begin{proposition} \label{pr:k_id}
Let $\lambda$ be a $4$-net of order $n$ embedded in $PG(2,\mathbb{K})$.  Then the cross-ratio $\kappa$ of $\lambda$ is an $N$--th root of unity of $\mathbb{K}$ such that $N=n(n-1)$ and
\begin{equation}
\label{eq9novbis}
(\kappa-1)^{N}=1.
\end{equation}
\end{proposition}
\begin{proof}
We prove first that $\kappa^{N}=1$. Let $P_{ij}$ be the common point of the lines $r_i$ and $w_j$ with $1\le i,j \leq n$. Then the unique line from $\lambda_3$ through $P_{ij}$ has equation $\sigma_{ij}r_i+\tau_{ij}w_j$ with $\sigma_{ij},\tau_{ij}\in \mathbb{K}^*$. Moreover, for any $k=1,\ldots,n$ there is a unique index $j$ such that $t_k= \sigma_{ij}r_i+\tau_{ij}w_j$. For every $i=1,\ldots,n$,
\begin{equation} \label{eqq3}
\alpha r_1\cdots r_n+\beta w_1\cdots w_n+\gamma[(\sigma_{i1}r_i+\tau_{i1}w_1)\cdots (\sigma_{in}r_i+\tau_{in}w_n)]=0.
\end{equation}
Take a point $Q$ on the line $r_i=0$ such that $w_j(Q)\neq 0$ for every $1\le j \le n$. Then
\[w_1(Q)\cdots w_n(Q)(\beta+\gamma \prod_{j=1}^n\tau_{ij})=0\]
yields
\begin{equation} \label{eqq4}
-\frac{\beta}{\gamma}= \prod_{j=1}^n\tau_{ij}
\end{equation}
for any fixed index $i$. The above argument applies to any line $w_j$ and gives
\begin{equation} \label{eqq5}
-\frac{\alpha}{\gamma}= \prod_{i=1}^n\sigma_{ij}
\end{equation}
for any fixed index $j$. Therefore,
\begin{equation} \label{eqq6}
\left(\frac{\beta}{\alpha}\right)^n= \prod_{i=1}^n\prod_{j=1}^n \frac{\tau_{ij}}{\sigma_{ij}}.
\end{equation}
A similar argument can be carried out for $\lambda_4$. The unique line from $\lambda_4$ through $P_{ij}$ has equation $\delta_{ij}r_i+\omega_{ij}w_j$ with $\delta_{ij},\omega_{ij}\in \mathbb{K}^*$. Then
\begin{equation} \label{eqq7}
\left(\frac{\beta'}{\alpha'}\right)^n= \prod_{i=1}^n\prod_{j=1}^n \frac{\omega_{ij}}{\delta_{ij}}.
\end{equation}
{}From Lemma \ref{pr:constant},
\[\frac{\tau_{ij}}{\sigma_{ij}}\cdot \frac{\delta_{ij}}{\omega_{ij}}=\kappa\]
for every $1\le i,j\le n$.
Then Equations (\ref{eqq6}) and (\ref{eqq7}) yield $\kappa^n=\kappa^{n^2}$ whence
\begin{equation}
\label{eqnov10}
\kappa^N=1.
\end{equation}
{}From the discussion at the beginning of this section, Equation (\ref{eqnov10}) holds true when $\kappa$ is replaced with any of the other five cross-ratio values. Therefore,  (\ref{eq9novbis}) also holds.
\end{proof}
In the complex plane, the cross-ratio equation has only two solutions, namely the roots of (\ref{eq9nov}). In fact, let $\kappa=x+yi$ with $x,y\in \mathbb{R}$. Then with respect to the complex norm,   (\ref{eqnov10}) and (\ref{eq9novbis}) imply $|x+iy|=x^2+y^2=1$ and $|x-1+iy|=(x-1)^2+y^2=1$. It hence follows that $\kappa=\frac{1}{2}(1\pm\sqrt{3}i)$, or equivalently (\ref{eq9nov}).

To extend this result to any field of characteristic zero, and discuss the positive characteristic case, look at
$$f(X)=\frac{X^N-1}{X-1}\,\,{\mbox{and}}\,\, g(X)=\frac{(X-1)^N-1}{X}$$ as polynomials in $\mathbb{Z}[X]$. From the preceding discussion on the complex case, their maximum common divisor is either $X^2-X+1$, or $1$ according as $6$ divides $N$ or does not. In the former case, divide both by $X^2-X+1$ and then replace $f(X)$ and $g(X)$ by them accordingly. Now, $f(X)$ and $g(X)$ are coprime, and hence their resultant is a non-zero integer $R$. Using a basic formula on resultants, see \cite[Lemma 2.3]{hkt}, $R$ may be computed in terms of a primitive $N$-th root of unity $\xi$, namely
$$ R=\prod_{1\leq i,j\leq N-1}(1+\xi^i-\xi^j),\, {\mbox{when}}\, 6\nmid N,$$ and
$$R=\prod_{\substack{1\leq i,j\leq N-1\\ i,j\neq N/6,\, 5N/6}} (1+\xi^i-\xi^j),\, {\mbox{when}}\, 6\mid N,$$
hold in the $N$-th cyclotomic field $\mathbb{Q}(\xi)$. Therefore, $R\neq 0$ provided that $\mathbb{K}$ has zero characteristic.
\begin{theorem}
\label{thenov10ter} Let $\mathbb{K}$ be a field of characteristic $0$. If a $4$-net $\lambda$ is embedded in $PG(2,\mathbb{K})$ then  $-3$ is a square in $\mathbb{K}$ and the constant cross-ratio $\kappa$ of $\lambda$ satisfies (\ref{eq9nov}).
\end{theorem}

To investigate the positive characteristic case,
we will use the well known result that $\mathbb{Q}(\xi)$ is a cyclic Galois extension of $\mathbb{Q}$ of degree $\varphi(N)$ where $\varphi$ is the classical Euler function. Let $\alpha$ be a generator of the Galois group. Then $\alpha(\xi)=\xi^m$ for a positive integer $m$ prime to $N$. Therefore, $\alpha$ permutes the factors in the right hand side. Given such a factor $1+\xi^i-\xi^j$, its cyclotomic
norm $$\parallel 1+\xi^i-\xi^j\parallel=(1+\xi^i-\xi^j)\cdot (1+\xi^i-\xi^j)^\alpha\cdot \ldots \cdot (1+\xi^i-\xi^j)^{\alpha^{N-1}}$$ is in $\mathbb{Q}$. Actually, it is an integer since the factors  are algebraic integers. Hence, the prime divisors of $R$ come from the prime divisors of the norms $\parallel 1+\xi^i-\xi^j\parallel$. Therefore, to find an upper bound on the largest prime divisor of  $R$ it is enough to find an upper bound on these norms. Obviously,

$$|\parallel 1+\xi^i-\xi^j\parallel|\leq |1+\xi^i-\xi^j|\cdot |1+\xi^{im}-\xi^{jm}|\cdot \ldots \cdot |1+\xi^{i(\varphi(N)-1))}-\xi^{j(\varphi(N)-1))}|.$$
Since $|1-\xi^i+\xi^j|\leq 3$, this shows that $|\parallel 1+\xi^i-\xi^j\parallel|\leq 3^{\varphi(N)}$. Hence the largest prime divisor of $R$ is at most $3^{\varphi(N)}$. Therefore, the following  result is proven.
\begin{theorem}
\label{th10nov} Let $\mathbb{K}$ be a field of characteristic $p>0$. If $p>3^{\varphi(n^2-n)}$  then Theorem \ref{thenov10ter} holds.
\end{theorem}
For planes over finite fields, Equations (\ref{eqnov10}) and (\ref{eq9novbis}) may provide further non-existence results on embedded $4$-nets.
\begin{theorem}
\label{thenov12} Let $\mathbb{K}=\mathbb{F}_q$ be a finite field of order $q=p^h$ with $p$ prime. If $p\neq 3$, then there exists no $4$-net of order $n$ embedded in $PG(2,\mathbb{F}_q)$ for $${\rm{gcd}}(n(n-1),q-1)\le 2.$$
\end{theorem}
\begin{proof} From Equation (\ref{eqnov10}), either $\kappa=1$ and $p=2$ or $\kappa^2=1$ and $p>2$. On the other hand, $\kappa\neq 1$. Hence $\kappa=-1$ and $p>2$. Now, Equation (\ref{eq9novbis}) yields $p=3$, a contradiction.
\end{proof}
The following example shows that the hypothesis $p\neq 3$ in Theorem \ref{thenov12} is essential.
\begin{example}
\label{ex3} Let $q=3^r$, and regard $PG(2,\mathbb{F}_q)$ as the projective closure of the affine plane $AG(2,\mathbb{F}_q)$. The four line sets $\lambda_1,\lambda_2,\lambda_3,\lambda_4$ form a $4$-net of order $q$ embedded in $AG(2,\mathbb{F}_q)$ where $\lambda_1$ and $\lambda_2$  consist of all horizontal and vertical lines respectively, while $\lambda_3$ and $\lambda_4$ consist of all lines with slope $1$ or $-1$, respectively. The constant cross-ratio of this $4$-net equals $-1$.
\end{example}

\section{Nets with more than four components}
We prove the non-existence of $5$-nets embedded in $PG(2,\mathbb{K})$ over a field $\mathbb{K}$  of characteristic $0$. This result was previously proved by Stipins \cite{sj2004}; see also \cite{ys2007}. Those authors used results and techniques  from Algebraic geometry. Here, we present a simple combinatorial proof depending on Theorem \ref{thenov10ter}. Our proof also works in positive characteristic $p$ whenever $p$ is big enough compared to the order $n$ of $4$-net; for example, when $p>3^{\varphi(n^2-n)}$ so that Theorem \ref{th10nov} holds. However, the non-existence result fails in general. This will be illustrated by means of some examples.

We begin with a technical lemma.
\begin{lemma} \label{lm:k5}
Let $A,B,C,D,D'$ be collinear points in $PG(2,\mathbb{K})$ with cross-ratios $\kappa=(ABCD)$ and $1-\kappa=(ABCD')$. If (\ref{eq9nov}) holds then $(ABDD')=-\kappa.$
\end{lemma}
\begin{proof}
Without loss of generality,  $A=(1,0,0)$, $B=(0,1,0)$, $C=(1,1,0)$. Then $D=(\kappa,1,0)$, $D'=(1-\kappa,1,0)$, and the result follows by a direct computation.
\end{proof}

\begin{theorem}
\label{fotetel}
 If the characteristic of the field $\mathbb{K}$ is either $0$ or greater than $3^{\varphi(n^2-n)}$, then there exists no $5$-net of order $n$ embedded in $PG(2,\mathbb{K})$.
\end{theorem}
\begin{proof}
Let $\Lambda=(\lambda_1,\lambda_2,\lambda_3,\lambda_4,\lambda_5)$ be a $5$-net of order $n$ embedded in $PG(2,\mathbb{K})$. Then $\Lambda_5=(\lambda_1,\lambda_2,\lambda_3,\lambda_4)$, $\Lambda_4=(\lambda_1,\lambda_2,\lambda_3,\lambda_5)$, and $\Lambda_{45}=(\lambda_1,\lambda_2,\lambda_4,\lambda_5)$
are three different $4$-nets and so we can compare their cross-ratios, say
$$\kappa_5=(l_1,l_2,l_3,l_4),\, \kappa_4=(l_1,l_2,l_3,l_5),\, \kappa_{45}=(l_1,l_2,l_4,l_5),$$ for five lines from different components and concurrent at a point of $\Lambda$.
{}From Proposition \ref{pr:k_id} each of them is a root of the polynomial $X^2-X+1$. Since $\Lambda_5$ and $\Lambda_4$ only differ in the last component,  $\kappa_5\neq \kappa_4$. Therefore, $\kappa_4=1-\kappa_5$. From Lemma \ref{lm:k5}, $\kappa_{45}=-\kappa_5$. This shows that $\kappa_{45}$ is not a root of $X^2-X+1$ contradicting
Proposition \ref{pr:k_id}.
\end{proof}

Example \ref{ex3} can be generalized for finite fields $\mathbb{F}_q$ with $q=p^r$ showing that $k$-nets arise from affine subplanes of $PG(2,\mathbb{K})$. Such $k$-nets have order $p^h$ with $h|r$. Here, we give further $k$-nets of $p$-power order. The construction relies on an idea of G. Lunardon \cite{Lunardon}. For the sake of simplicity, we describe the construction in terms of a dual $k$-net, that is, the components are sets of points such that a line connecting two points of different components hits any third component in precisely one point.

\begin{example} \label{exam:Lun}
Let $\mathbb{K}=\mathbb{F}_q$ such that $q=r^s$ with $s\geq 3$. Take elements $u,v\in \mathbb{F}_q$ such that $1,u,v$ are linearly independent over the subfield $\mathbb{F}_r$. Take a basis $\mathbf{b}_1,\mathbf{b}_2$ of $\mathbb{F}_q^2$ and put $\mathbf{b}_0=u\mathbf{b}_1 + v\mathbf{b}_2$. For any $\alpha \in \mathbb{F}_r$, we define the points sets
\[A_\alpha=\{\alpha \mathbf{b}_0 + \lambda \mathbf{b}_1 + \mu \mathbf{b}_2 \mid \lambda, \mu \in \mathbb{F}_r \} \]
in $AG(2,q)$. Then the $A_\alpha$'s ($\alpha \in \mathbb{F}_r$) are components of a dual $r$-net of order $r^2$. In order to see this, take the points
\[P_i=\alpha_i \mathbf{b}_0 + \lambda_i \mathbf{b}_1 + \mu_i \mathbf{b}_2, \hspace{1cm} i=1,2,3.\]
$P_1,P_2,P_3$ are collinear in $AG(2,q)$ if and only if the vectors
\begin{equation} \label{eq:paral}
(\alpha_1-\alpha_2) \mathbf{b}_0 + (\lambda_1-\lambda_2) \mathbf{b}_1 + (\mu_1-\mu_2) \mathbf{b}_2 \mbox{ and }
(\alpha_1-\alpha_3) \mathbf{b}_0 + (\lambda_1-\lambda_3) \mathbf{b}_1 + (\mu_1-\mu_3) \mathbf{b}_2
\end{equation}
are linearly dependent over $\mathbb{F}_q$. By the definition of $\mathbf{b}_0$ and the independence of $\mathbf{b}_1, \mathbf{b}_2$, \eqref{eq:paral} is equivalent with
\begin{multline} \label{eq:det}
((\alpha_1-\alpha_2)u + \lambda_1-\lambda_2)((\alpha_1-\alpha_3)v + \mu_1-\mu_3) - \\
((\alpha_1-\alpha_3)u + \lambda_1-\lambda_3)((\alpha_1-\alpha_2)v + \mu_1-\mu_2) = 0.
\end{multline}
Sorting by $u$ and $v$, we obtain
\begin{eqnarray*} 
0&=&u [(\alpha_1-\alpha_2)(\mu_1-\mu_3)-(\alpha_1-\alpha_3)(\mu_1-\mu_2)] \\
&& + v [(\alpha_1-\alpha_3)(\lambda_1-\lambda_2) - (\alpha_1-\alpha_2)(\lambda_1-\lambda_3)] \\
&& + (\lambda_1-\lambda_2)(\mu_1-\mu_3) - (\mu_1-\mu_2)(\lambda_1-\lambda_3).
\end{eqnarray*}
The independence of $1,u,v$ over $\mathbb{F}_r$ implies the system of equations
\begin{align}
0&=(\alpha_1-\alpha_2)(\mu_1-\mu_3)-(\alpha_1-\alpha_3)(\mu_1-\mu_2), \label{eq:sys1}\\
0&=(\alpha_1-\alpha_3)(\lambda_1-\lambda_2) - (\alpha_1-\alpha_2)(\lambda_1-\lambda_3), \label{eq:sys2}\\
0&=(\lambda_1-\lambda_2)(\mu_1-\mu_3) - (\mu_1-\mu_2)(\lambda_1-\lambda_3). \label{eq:sys3}
\end{align}
With given points $P_1,P_2$, $\alpha_1\neq \alpha_2$, \eqref{eq:sys1} and \eqref{eq:sys2} has the unique solution
\begin{align*}
\lambda_3 &= \frac{\lambda_1(\alpha_3-\alpha_2) + \lambda_2(\alpha_1-\alpha_3)}{\alpha_1-\alpha_2}, \\
\mu_3 &= \frac{\mu_1(\alpha_3-\alpha_2) + \mu_2(\alpha_1-\alpha_3)}{\alpha_1-\alpha_2},
\end{align*}
which is a solution for \eqref{eq:sys3}, as well. This means that the line $P_1P_2$ hits $A_{\alpha_3}$ in the unique point
\[P_3=\frac{\alpha_3-\alpha_2}{\alpha_1-\alpha_2} P_1 + \frac{\alpha_1-\alpha_3}{\alpha_1-\alpha_2} P_2.\]
This formula further shows that the constant cross-ratio can take any value in $\mathbb{F}_r\setminus \{0,1\}$.
\end{example}
We are able to describe the geometric structure of $k$-nets ($k\geq 4$) where one component is contained in a line pencil.
\begin{theorem}
Let $\lambda=(\lambda_1,\ldots,\lambda_k)$, $k\geq 4$, be a $k$-net of order $n$ embedded in $PG(2,\mathbb{K})$. Assume that the component $\lambda_1$ is contained in a line pencil. Then the following hold.
\begin{enumerate}
\item The order of $\lambda$ is $n=p^e$ where $p>0$ is the characteristic of $\mathbb{K}$.
\item For each component $\lambda_i$, $i>1$, there is an elementary Abelian $p$-group of collineations acting regularly on the lines of $\lambda_i$.
\item The components $\lambda_2,\ldots,\lambda_k$ are projectively equivalent.
\item If any other component is contained in a line pencil then all components are, and the base points of the pencils are collinear.
\end{enumerate}
\end{theorem}
\begin{proof}
It suffices to prove the theorem for $k=4$. We give the proof for the dual $k$-net by assuming that the component $\lambda_1$ is contained in the line $\ell$. Let $\kappa$ be the constant cross-ratio of $(\lambda_2,\lambda_3,\lambda_4,\lambda_1)$ and for any point $S\not\in \ell$ denote by $u_S$ the $(S,\ell)$-perspectivity such that for any point $P$ and its image $P'=u_S(P)$, the cross-ratio of $S,P,P'$ and $PP'\cap \ell$ is $\kappa$. Then, for any $S\in \lambda_2$, $u_S$ induces a bijection between $\lambda_3$ and $\lambda_4$. In particular, $\lambda_3$ and $\lambda_4$ are projectively equivalent. Let $S,T \in \lambda_2$, $S\neq T$, and assume that $u_S^{-1}u_T$ has a fixed point $R\not\in \ell$, that is, $u_S(R)=u_T(R)=R'$. Then, $S,T\in RR'$ and with $R''=RR'\cap \ell$ the cross-ratios $(S,R,R',R'')$, $(T,R,R',R'')$  are equal to $\kappa$. This implies $S=T$, a contradiction. This means that for all $S,T \in \lambda_2$, $S\neq T$, the collineation $u_S^{-1}u_T$ is an elation with axis $\ell$, and $\{ u_S^{-1}u_T \mid S,T \in \lambda_2\}$ generate an elementary Abelian $p$-group $U$ of collineations, leaving $\lambda_3$ invariant. Moreover, $U$ acts transitively, hence regularly on $\lambda_3$. This finishes the proof.
\end{proof}

\begin{example}
In Example \ref{exam:Lun}, we constructed a dual $r$-net of order $r^2$ in $AG(2,r^s)$, $s\geq 3$. For $P_1\in A_{\alpha_1}$, $P_2\in A_{\alpha_2}$, the line $P_1P_2$ has direction vectors
\[(u+\lambda)\mathbf{b}_1+(v+\mu)\mathbf{b}_2.\]
These are linearly independent for different choices of $\lambda,\mu \in \mathbb{F}_r$, hence they determine $r^2$ points at infinity. Let $\lambda_0$ be the set of corresponding infinite points. Then, $(\lambda_0,\lambda_1,\ldots,\lambda_r)$ is a dual $(r+1)$-net with component $\lambda_0$ contained in a line.
\end{example}

\vspace{0,5cm}\noindent {\em Authors' addresses}:

\vspace{0.2cm}\noindent G\'abor KORCHM\'AROS\\ Dipartimento di
Matematica e Informatica\\ Universit\`a della Basilicata\\ Contrada Macchia
Romana\\ 85100 Potenza (Italy)
\\E--mail: {\tt gabor.korchmaros@unibas.it }

\vspace{0.2cm}\noindent G\'abor P. NAGY\\
Bolyai Institute \\ University of Szeged \\
Aradi v\'ertan\'uk tere 1\\ 6725 Szeged (Hungary)\\
E--mail: {\tt nagyg@math.u-szeged.hu }

\vspace{0.2cm}\noindent Nicola PACE\\
Inst. de Ci\^{e}ncias Matem\'{a}ticas e de Computa\c{c}\~{a}o  \\
Universidade de S\~{a}o Paulo \\
Av. do Trabalhador S\~{a}o-Carlense, 400 \\
S\~{a}o Carlos, SP 13560-970, Brazil  \\
E--mail: {\tt nicolaonline@libero.it}
\end{document}